\documentclass[11pt]{article}
\usepackage{amssymb,amsmath,amsthm,amsfonts}
\usepackage{graphicx}
\usepackage{epsfig}

\usepackage{amssymb,amsmath,amsthm,amsfonts,mathrsfs, bbm,dsfont}

\theoremstyle{plain}

\theoremstyle{definition}
\newtheorem{definition}{Definition}[section]
\newtheorem{example}{Example}[section]
\newtheorem{remark}{Remark}[section]

\newtheorem{theorem}{Theorem}[section]
\newtheorem{lemma}{Lemma}[section]

\newtheorem{corollary}{Corollary}[section]

\numberwithin{equation}{section}

\begin{document}
\openup 0.8\jot
\title{\Large\bf Common fixed point theorems in $C^*$-algebra-valued metric spaces \thanks{This work is supported
by National Science Foundation of China (10971011,11371222)} }
\author{Xin Qiaoling $^{1}$, Jiang Lining $^{1}$ \thanks{E-mail address: jianglining@bit.edu.cn}, Ma Zhenhua $^{1,2}$}
\date{}
\maketitle\begin{center}
\begin{minipage}{16cm}
{\small 1. \it School of Mathematics and Statistics, Beijing Institute of Technology, Beijing, 100081, P. R. China

 2. \it Department of Mathematics, Institute of Architecture Civil Engineering, Zhangjiakou, 075000, P. R. China}
\end{minipage}
\end{center}
\vspace{0.05cm}
\begin{center}
\begin{minipage}{16cm}
{\small {\bf Abstract}:
In this paper, we establish coincidence fixed point and common fixed point theorems for two mappings in complete $C^*$-algebra-valued metric spaces which satisfy new contractive conditions. Some applications of our obtained results are given.
}
\endabstract
\end{minipage}\vspace{0.10cm}
\begin{minipage}{16cm}
\end{minipage}
\end{center}
\begin{center} \vspace{0.01cm}
\end{center}







\section{Introduction}
The Banach contraction principle, which shows that every contractive mapping has a unique fixed point in a complete metric space,
has been extended in many directions ([1-15,17-22]).
One of the branches of this theory is devoted to the study of common fixed points.
In 1966, Jungck \cite{Jungck66} initially investigated common fixed points for commuting mappings in metric spaces.
The concept of
commuting mappings has been weakened in various directions and in several ways over the years.
One such notion which is weaker than commuting is the concept of compatibility introduced by Jungck \cite{Jungck86}.
Subsequently, several authors have obtained coincidence and common fixed point results for mappings, utilizing this concept and its generalizations, see \cite{M.Abb,Beri,Ciri,Jankovi,Jungck09,Shata} and references contained therein.

On the other hand, there are a lot of fixed and common fixed point results in different types of spaces. For
example, cone metric spaces \cite{Huang}, fuzzy metric spaces \cite{Abu}, uniform spaces \cite{Tarafd}, noncommutative Banach spaces \cite{Xin}, and so on.
In 2007, Huang and Zhang \cite{Huang} firstly introduced cone metric spaces which generalized metric spaces, and obtained various fixed point theorems for contractive mappings. The existence of a common fixed point on cone metric spaces was investigated recently in \cite{M.Abb,Choudh,Jankovi,Jungck09}. Very recently, Ma et al. in \cite{Ma} introduced a concept
of $C^*$-algebra-valued metric spaces and presented some fixed point results for mappings under contractive or expansive conditions in these spaces.

In this paper, we will continue to study common fixed points in the frame of $C^*$-algebra-valued metric spaces. More precisely, we
prove some common fixed point theorems for two mappings under a different contractive conditions .
We furnish suitable examples to demonstrate the validity of
the hypotheses of our results. The presented theorems extend and improve some recent results given in \cite{Ma}. In addition,
we establish the existence and uniqueness theorem of a common solution for integral equations.

Throughout this paper, the letter $\mathcal A$ will denote an unital $C^*$-algebra. Set $\mathcal A_h\ =\ \{a\in \mathcal A\colon x=x^*\}$.
We call an element $x\in \mathcal A$ a positive element, denote it by
$0_{\mathcal A}\preceq a$, if $x=x^*$
and $\sigma(x)\subseteq [0,+\infty)$, where $0_{\mathcal A}$ is the zero element in $\mathcal A$ and $\sigma(x)$ is the spectrum of $x$. There is a natural partial ordering on $\mathcal A_h$ given by
$a \preceq b$ if and only if $0_{\mathcal A}\preceq b-a$. From now on, ${\mathcal A}_+$ and ${\mathcal A}'$ will denote the set $\{a\in \mathcal A\colon 0_{\mathcal A}\preceq a\}$ and the set $\{a\in \mathcal A\colon ab=ba, \forall b\in\mathcal A\}$, respectively.

Let us recall the following definitions and results which will be needed in what follows. For more details, one can see \cite{Ma}.

\begin{definition}
Let $X$ be a nonempty set. Suppose that the mapping $d\colon X\times X\rightarrow \mathcal A$ is defined, with the following properties:

(1)\ $0_{\mathcal A}\preceq d(x,y)$ for all $x$ and $y$ in $X$;

(2)\ $d(x,y)=0_{\mathcal A}$ if and only if $x=y$;

(3)\ $d(x, y)=d(y, x)$ for all $x$ and $y$ in $X$;

(4)\ $d(x, y)\preceq d(x, z)+d(z, y)$ for all $x$, $y$ and $z$ in $X$.\\
Then $d$ is said to be a $C^*$-algebra-valued metric on $X$, and $(X,{\mathcal A},d)$ is said to be a $C^*$-algebra-valued metric space.
\end{definition}

\begin{definition}
Suppose that $(X,{\mathcal A},d)$ is a $C^*$-algebra-valued metric space. Let $\{x_n\}_{n=1}^\infty$ be a sequence in $X$ and $x\in X$. If
$d(x_n, x)\overset{\|\cdot\|_{\mathcal A}}{\longrightarrow} 0_{\mathcal A}$ $(n\rightarrow\infty)$, then it is said that $\{x_n\}$ converges to $x$, and we denote it by $\lim\limits_{n\rightarrow\infty}x_n=x$.
If for any $p\in\Bbb{N}$,
$d(x_{n+p}, x_n)\overset{\|\cdot\|_{\mathcal A}}{\longrightarrow} 0_{\mathcal A}$ $(n\rightarrow\infty)$, then $\{x_n\}$ is called a
Cauchy sequence in $X$.
\end{definition}

If every Cauchy sequence is convergent in $X$, then
$(X,{\mathcal A},d)$ is called a complete $C^*$-algebra-valued metric space.

It is obvious that any Banach space must be a complete $C^*$-algebra-valued metric space. Moreover,
$C^*$-algebra-valued metric spaces generalize normed linear spaces and metric spaces.

\begin{definition}
Let $Y$ be a subset of $X$. If $(Y,{\mathcal A},d)$ is a complete $C^*$-algebra-valued metric space, then we say that $Y$ is complete in $X$.
\end{definition}

\begin{example}
Let $X=\Bbb{R}$ and $\mathcal{A}=M_2(\Bbb{C})$,
the set of bounded linear operators on a Hilbert space $\Bbb{C}^2$. Define $d\colon X\times X\rightarrow \mathcal{A}$ by
$d(x,y)=\left[
          \begin{array}{cc}
            |x-y| & 0 \\
            0 & k|x-y| \\
          \end{array}
        \right]$,
where $k>0$ is a constant.
Then $(X,\mathcal{A},d)$ is a complete $C^*$-algebra-valued metric space. If we choose $Y=[0,1]\subseteq X$, we can show $Y$ is complete in $X$.
If $Y=(-\infty,0)\cup(0,+\infty)$, then $Y$ is not complete in $X$. Indeed, taking $\{x_n\}\subseteq Y$ such that $x_n=\frac{1}{n}$, we get
$d(x_n,0)=\left[
          \begin{array}{cc}
            \frac{1}{n} & 0 \\
            0 & \frac{k}{n} \\
          \end{array}
        \right]\overset{\|\cdot\|_{\mathcal A}}{\longrightarrow} 0_{\mathcal A}$,
        which means $x_n\rightarrow 0\notin Y$ $(n\rightarrow\infty)$.
\end{example}

\begin{lemma}\label{L1}
(1)\ If $\{b_n\}_{n=1}^\infty\subseteq\mathcal A$ and $\lim\limits_{n\rightarrow\infty}b_n=0_{\mathcal A}$, then for any $a\in \mathcal A$, $\lim\limits_{n\rightarrow\infty}a^*b_na=0_{\mathcal A}$.

(2)\ If $a,b\in{\mathcal A}_h$ and $c\in {\mathcal A}_+'$, then
$a\preceq b$ deduces $ca\preceq cb$, where ${\mathcal A}_+'={\mathcal A}_+\cap {\mathcal A}'$.

(3)\ Let $\{x_n\}_{n=1}^\infty$ be a sequence in $X$. If $\{x_n\}$ converges to $x$ and $y$, respectively, then $x=y$. That is,
the limit of a convergent sequence in a $C^*$-algebra-valued metric space is unique.
\end{lemma}

\begin{proof}
(1)\ By the following relation $\|a^*b_na-0_{\mathcal A}\|\leqslant\|a\|^2\|b_n\|$,
we can get the desired result.

(2)\ $a\preceq b$ implies $b-a\in{\mathcal A}_+$, and then there is $d\in {\mathcal A}_+$ such that $b-a=d^2$. Again,
$c\in{\mathcal A}_+'$, then $c=e^2$ for some $e\in{\mathcal A}_+$.
Note that
\begin{equation*}
cb-ca=c(b-a)=e^2d^2=eded=(ed)^*ed\in{\mathcal A}_+,
\end{equation*}
which shows
$ca\preceq cb$.

(3)\ Using the triangle inequality, we get
\begin{equation*}
d(x,y)\preceq d(x_n,x)+d(x_n,y),
\end{equation*}
which, together with
$\lim\limits_{n\rightarrow\infty}x_n=x$ and $\lim\limits_{n\rightarrow\infty}x_n=y$, deduces
$d(x,y)\overset{\|\cdot\|_{\mathcal A}}{\longrightarrow} 0_{\mathcal A}$ $(n\rightarrow\infty)$. Hence $d(x,y)=0_{\mathcal A}$, and then $x=y$.
\end{proof}

\begin{remark}
In Lemma 1.1 (2), the element $c\in {\mathcal A}_+'$ is necessary.
For example, let $a=\left[
          \begin{array}{cc}
            0 & 3 \\
            3 & 1 \\
          \end{array}
        \right],$
$b=\left[
          \begin{array}{cc}
            1 & 1 \\
            1 & 6 \\
          \end{array}
        \right]$
and $c=\left[
          \begin{array}{cc}
            1 & 1 \\
            1 & 1 \\
          \end{array}
        \right].$
One can show that

(i)\ $a,b\in{\mathcal A}_h$, with $a\preceq b$ since
$b-a=\left[
          \begin{array}{cc}
            1 & -2 \\
            -2 & 5 \\
          \end{array}
        \right]\in{\mathcal A}_+$.

(ii)\ $c\in{\mathcal A}_+$ but $c\notin{\mathcal A}'$.

(iii) Since
$cb-ca=\left[
          \begin{array}{cc}
            -1 & 3 \\
            -1 & 3 \\
          \end{array}
        \right]$, we know that $cb-ca\notin{\mathcal A}_+$.
\end{remark}

The following definition extends the concept
of compatible mappings of Jungck \cite{Jungck86}, from metric spaces to $C^*$-algebra-valued metric spaces.

\begin{definition}
The two mappings $T$ and $S$ on a $C^*$-algebra-valued metric space $(X,\mathcal{A},d)$ is said to be compatible, if
for arbitrary sequence $\{x_n\}_{n=1}^\infty\subseteq X$, such that
$\lim\limits_{n\rightarrow\infty}Tx_n
=\lim\limits_{n\rightarrow\infty}Sx_n=t\in X$,
then
$d(TSx_n,STx_n)\overset{\|\cdot\|_{\mathcal A}}{\longrightarrow} 0_{\mathcal A}$ $(n\rightarrow\infty)$.
\end{definition}

\begin{definition}
Let $T$ and $S$ be two mappings of the set $X$.

(1)\ If $x=Tx=Sx$ for some $x\in X$, then $x$ is called a common fixed point of $T$ and $S$.

(2)\ If $z=Tx=Sx$ for some $z\in X$, then $x$ is called a coincidence point of $T$ and $S$, and $z$ is called a point of coincidence of $T$ and $S$.

(3)\ If $T$ and $S$ commute at all of their coincidence points, i.e.,
$TSx=STx$ for all $x\in \{x\in X\colon Tx=Sx\}$, then $T$ and $S$ are called weakly compatible.
\end{definition}

In metric spaces if the mappings $T$ and $S$ are compatible, then they are weakly compatible, while the converse is not
true \cite{Jungck86}. The same holds for the $C^*$-algebra-valued metric spaces:

\begin{lemma}
If the mappings $T$ and $S$ on the $C^*$-algebra-valued metric space $(X,\mathcal{A},d)$ are compatible, then they are weakly compatible.
\end{lemma}

\begin{proof}
Let $Tx=Sx$ for some $x\in X$. It suffices to show that $TSx=STx$.
Putting $x_n\equiv x$ for every $n\in\Bbb{N}$, we have $\lim\limits_{n\rightarrow\infty}Tx_n
=\lim\limits_{n\rightarrow\infty}Sx_n$, and then, since
$T$ and $S$ are compatible, we have
$d(TSx_n,STx_n)\overset{\|\cdot\|_{\mathcal A}}{\longrightarrow} 0_{\mathcal A}$ $(n\rightarrow\infty)$, that is,
$\|d(TSx_n,STx_n)\|\longrightarrow 0$ $(n\rightarrow\infty)$.
Hence $d(TSx,STx)=0_{\mathcal A}$, which means $TSx=STx$.
\end{proof}

The converse does not hold. For example,
let $X=[0,4]$ and ${\mathcal A}=M_2(\Bbb{C})$. Define $d\colon X\times X\rightarrow \mathcal{A}$ by
$d(x,y)=\left[
          \begin{array}{cc}
            |x-y| & 0 \\
            0 & k|x-y| \\
          \end{array}
        \right]$,
where $k\geqslant 0$ is a constant.
Then $(X,\mathcal{A},d)$ is a $C^*$-algebra-valued metric space.
Set
\begin{eqnarray*}
Tx=\left\{\begin{array}{cl}
            3-x & x\in [0,\frac{3}{2}], \\[5pt]
            3 & x\in (\frac{3}{2},4],
          \end{array}\right.\ \ \ \
Sx=\left\{\begin{array}{cl}
             2x & x\in (1,2],\\[5pt]
              x & x\in [0,1]\cup(2,4].
          \end{array}\right.
\end{eqnarray*}
Firstly, we can compute that
the set of their coincidence points is singleton set $\{3\}$, and then we have $T$ and $S$ commute at this point.
Hence, $T$ and $S$ are weakly compatible.
However, we can show they are not compatible. In order to do this, we construct a sequence $\{x_n\}\subseteq X$ such that $x_n=1+\frac{1}{n}\in X$
for $n\in \Bbb{N}$ with $n\geqslant 2$. In this case, we have
$Tx_n=3-(1+\frac{1}{n})=2-\frac{1}{n}$,
and $Sx_n=2(1+\frac{1}{n})=2+\frac{2}{n}$.
Then
$\lim\limits_{n\rightarrow\infty}Tx_n
=\lim\limits_{n\rightarrow\infty}Sx_n=2$. In fact,
we have
\begin{equation*}
d(Tx_n,2)=d(2-\frac{1}{n},2)
=\left[
          \begin{array}{cc}
           \frac{1}{n} & 0 \\
            0 & k\frac{1}{n} \\
          \end{array}
        \right]
\overset{\|\cdot\|_{\mathcal A}}{\longrightarrow} 0_{\mathcal A}\ \ (n\rightarrow\infty)
\end{equation*}
and
\begin{equation*}
d(Sx_n,2)=d(2+\frac{2}{n},2)
=\left[
          \begin{array}{cc}
            \frac{2}{n} & 0 \\
            0 & k\frac{2}{n} \\
          \end{array}
        \right]
\overset{\|\cdot\|_{\mathcal A}}{\longrightarrow} 0_{\mathcal A}\ \ (n\rightarrow\infty).
\end{equation*}
But
\begin{equation*}
d(TSx_n,STx_n)=d(T(2+\frac{2}{n}),S(2-\frac{1}{n}))
=d(3,4-\frac{2}{n})
=\left[
          \begin{array}{cc}
            |1-\frac{2}{n}| & 0 \\
            0 & k|1-\frac{2}{n}| \\
          \end{array}
        \right]
\overset{\|\cdot\|_{\mathcal A}}{\longrightarrow}
\left[
          \begin{array}{cc}
            1 & 0 \\
            0 & k \\
          \end{array}
        \right],
\end{equation*}
which means that
$d(TSx_n,STx_n)\overset{\|\cdot\|_{\mathcal A}}
{\nrightarrow}0_{\mathcal A}.$

The following lemma can be seen in \cite{M.Abb}.
\begin{lemma}\label{L2}
Let $T$ and $S$ be weakly compatible mappings of a set $X$. If $T$ and $S$ have a unique point of coincidence, then it is the unique common fixed point of $T$ and $S$.
\end{lemma}

\section{Main results}
In this section, we give some common fixed point theorems for two mappings satisfying various contractive conditions in complete $C^*$-algebra-valued metric spaces.

\begin{theorem}\label{Th1}
Let $(X,\mathcal{A},d)$ be a complete $C^*$-algebra-valued metric space.
Suppose that two mappings $T,S\colon X\rightarrow X$ satisfy
\begin{equation}\label{2-1}
d(Tx,Sy)\preceq a^*d(x,y)a, \ \mbox{for\ any}\ x,y\in X,
\end{equation}
where $a\in \mathcal{A}$ with $\|a\|<1$. Then $T$ and $S$ have a unique
common fixed point in $X$.
\end{theorem}

\begin{proof}
Let $x_0\in X$ and construct a sequence $\{x_n\}_{n=0}^\infty\subseteq X$ by the way: $x_{2n+1}=Tx_{2n}$, $x_{2n+2}=Sx_{2n+1}$. From (\ref{2-1}), we get
\begin{eqnarray*}
    \begin{array}{rcl}
d(x_{2n+2},x_{2n+1})&=&d(Sx_{2n+1},Tx_{2n})\\[5pt]
      &\preceq&a^*d(x_{2n+1},x_{2n})a\\[5pt]
      &\preceq&(a^*)^2d(x_{2n},x_{2n-1})a^2\\[5pt]

      &&\cdots\\[5pt]
      &\preceq&(a^*)^{2n+1}d(x_1,x_0)a^{2n+1},
 \end{array}
\end{eqnarray*}
where we use the property: if $b,c\in {\mathcal A}_h$, then $b\preceq c$ implies $a^*ba\preceq a^*ca$ (Theorem 2.2.5 in \cite{Murphy}).

Similarly,
\begin{eqnarray*}
    \begin{array}{rcl}
d(x_{2n+1},x_{2n})&=&d(Tx_{2n},Sx_{2n-1})\\[5pt]
      &\preceq&a^*d(x_{2n},x_{2n-1})a\\[5pt]
      &&\cdots\\[5pt]
      &\preceq&(a^*)^{2n}d(x_1,x_0)a^{2n}.
 \end{array}
\end{eqnarray*}
Now, we can obtain for any $n\in\Bbb{N}$
\begin{equation*}
d(x_{n+1},x_{n})
      \preceq(a^*)^{n}d(x_1,x_0)a^{n},
\end{equation*}
then for any $p\in\Bbb{N}$, the triangle inequality tells that
\begin{eqnarray*}
    \begin{array}{rcl}
d(x_{n+p},x_{n})&\preceq&d(x_{n+p},x_{n+p-1})+d(x_{n+p-1},x_{n+p-2})
              +\cdots+d(x_{n+1},x_{n})\\[5pt]
      &\preceq&\sum\limits_{k=n}^{n+p-1}(a^*)^kd(x_1,x_0)a^k\\[5pt]
      &\preceq&\sum\limits_{k=n}^{n+p-1}
      (ba^k)^*ba^k\\[5pt]
      &\preceq&\sum\limits_{k=n}^{n+p-1}
      |ba^k|^2\\[5pt]
      &\preceq&\sum\limits_{k=n}^{n+p-1}
      \||ba^k|^2\|1_{\mathcal A}\\[5pt]
      &\preceq&\|b\|^2 1_{\mathcal A}\sum\limits_{k=n}^{n+p-1}
      \|a^k\|^2   \rightarrow 0_{\mathcal A} \ (n\rightarrow\infty),
 \end{array}
\end{eqnarray*}
where $1_{\mathcal A}$ is the unit element in ${\mathcal A}$ and $d(x_1,x_0)=b^2$ for some $b\in{\mathcal A}_+$, this can be done since $d(x_1,x_0)\in{\mathcal A}_+$ (Theorem 2.2.1 in \cite{Murphy}).

From Definition 1.2, we get that $\{x_n\}_{n=1}^\infty$ is a Cauchy sequence in $X$.
The completion of $X$ implies that there exists $x\in X$ such that $\lim\limits_{n\rightarrow\infty}x_n=x$.

Now, using the triangle inequality and (\ref{2-1}),
\begin{equation*}
\begin{array}{rcl}
d(x,Sx)&\preceq&d(x,x_{2n+1})+d(x_{2n+1},Sx)\\[5pt]
       &\preceq&d(x,x_{2n+1})+d(Tx_{2n},Sx)\\[5pt]
       &\preceq&d(x,x_{2n+1})+a^*d(x_{2n},x)a.
\end{array}
\end{equation*}
Taking $n\rightarrow\infty$, the right hand side of the above inequality approaches $0_\mathcal A$ (Lemma 1.1 (1)), and then $Sx=x$.
Again, noting that
\begin{equation*}
0_{\mathcal A}\preceq d(Tx,x)=d(Tx,Sx)\preceq a^*d(x,x)a=0_{\mathcal A},
\end{equation*}
we have $d(Tx,x)=0_{\mathcal A}$, which means $Tx=x$.

In the following we will show the uniqueness of common fixed points in $X$. For this purpose, assume that there is another point $y\in X$ such that $Ty=Sy=y$.
From (\ref{2-1}), we know
\begin{equation*}
d(x,y)=d(Tx,Sy)\preceq a^*d(x,y)a,
\end{equation*}
which together with $\|a\|<1$ yields that
\begin{equation*}
0\leqslant\|d(x,y)\|\leqslant \|a\|^2\|d(x,y)\|<\|d(x,y)\|.
\end{equation*}
Thus $\|d(x,y)\|=0$ and $d(x,y)=0_{\mathcal A}$,
which gives $y=x$. Hence, $T$ and $S$ have a unique
common fixed point in $X$.
\end{proof}

If one checks the proof of Theorem 2.1, one can easily obtain the following result.

\begin{corollary}
Let $(X,\mathcal{A},d)$ be a complete $C^*$-algebra-valued metric space.
Suppose that the mappings $T,S\colon X\rightarrow X$ satisfy
\begin{equation*}
\|d(Tx,Sy)\|\leqslant \|a\|\|d(x,y)\|, \ \mbox{for\ any}\ x,y\in X,
\end{equation*}
where $a\in \mathcal{A}$ with $\|a\|<1$. Then $T$ and $S$ have a unique
common fixed point in $X$.
\end{corollary}

\begin{corollary}
Let $(X,\mathcal{A},d)$ be a complete $C^*$-algebra-valued metric space.
Suppose that the mapping $T\colon X\rightarrow X$ satisfies
\begin{equation*}
d(T^mx,T^ny)\preceq a^*d(x,y)a, \ \mbox{for\ any}\ x,y\in X,
\end{equation*}
where $a\in \mathcal{A}$ with $\|a\|<1$, and $m$ and $n$ are fixed positive integers. Then $T$ has a unique fixed point in $X$.
\end{corollary}

\begin{proof}
By setting $T=T^m$ and $S=T^n$ in (\ref{2-1}),
the result then follows from Theorem 2.1.
\end{proof}

\begin{remark}
In Theorem \ref{Th1}, if $S=T$, (\ref{2-1}) becomes
\begin{equation}\label{2-2}
d(Tx,Ty)\preceq a^*d(x,y)a, \ \mbox{for\ any}\ x,y\in X,
\end{equation}
where $a\in \mathcal{A}$ with $\|a\|<1$.
In this case, we have the following corollary, which can also be found in \cite{Ma}.
\end{remark}

\begin{corollary}
Let $(X,\mathcal{A},d)$ be a complete $C^*$-algebra-valued metric space.
Suppose that the mapping $T\colon X\rightarrow X$ satisfies
(\ref{2-2}),
then $T$ has a unique fixed point in $X$.
\end{corollary}

\begin{theorem}
Let $(X,\mathcal{A},d)$ be a complete $C^*$-algebra-valued metric space.
Suppose that two mappings $T,S\colon X\rightarrow X$ satisfy
\begin{equation}\label{2-3}
d(Tx,Ty)\preceq a^*d(Sx,Sy)a, \ \mbox{for\ any}\ x,y\in X,
\end{equation}
where $a\in \mathcal{A}$ with $\|a\|<1$. If $R(T)$ is contained in $R(S)$ and $R(S)$ is complete in $X$,
then $T$ and $S$ have a unique point of coincidence in $X$. Furthermore,
if $T$ and $S$ are weakly compatible,
$T$ and $S$ have a unique
common fixed point in $X$.
\end{theorem}

\begin{proof}
Let $x_0\in X$. Choose $x_1\in X$ such that $Sx_1=Tx_0$, which can be done since $R(T)\subseteq R(S)$. Let $x_2\in X$ such that $Sx_2=Tx_1$.
Repeating the process, we get a sequence $\{x_n\}_{n=1}^\infty$ in $X$ satisfying $Sx_{n}=Tx_{n-1}$. Then from (\ref{2-3}),
\begin{eqnarray*}
    \begin{array}{rcl}
d(Sx_{n+1},Sx_{n})&=&d(Tx_{n},Tx_{n-1})\\[5pt]
      &\preceq&a^*d(Sx_{n},Sx_{n-1})a\\[5pt]
      &&\cdots\\[5pt]
      &\preceq&(a^*)^{n}d(Sx_1,Sx_0)a^{n},
 \end{array}
\end{eqnarray*}
which shows that $\{Sx_n\}_{n=1}^\infty$ is a Cauchy sequence in $R(S)$. Since $R(S)$ is complete in $X$, there exists $q\in X$ such that $\lim\limits_{n\rightarrow\infty}Sx_n=Sq$.
\begin{eqnarray*}
d(Sx_{n},Tq)=d(Tx_{n-1},Tq)\preceq a^*d(Sx_{n-1},Sq)a,
\end{eqnarray*}
From $\lim\limits_{n\rightarrow\infty}Sx_n=Sq$ and Lemma \ref{L1} (1), we get $a^*d(Sx_{n-1},Sq)a\rightarrow 0_{\mathcal A}$ as $n\rightarrow \infty$, and then $\lim\limits_{n\rightarrow \infty}Sx_{n}=Tq$.
It follows from Lemma \ref{L1} (3) that
$Tq=Sq$.
If there is a point $w$ in $X$ such that $Tw=Sw$, (\ref{2-3}) shows
\begin{eqnarray*}
d(Sq,Sw)=d(Tq,Tw)\preceq a^*d(Sq,Sw)a.
\end{eqnarray*}
The same reasoning that in Theorem \ref{Th1} tells us that $Sq=Sw$.
Hence, $T$ and $S$ have a unique point of coincidence in $X$. It follows from Lemma \ref{L2} that $T$ and $S$ have a unique common fixed point in $X$.
\end{proof}

\begin{example}
In Theorem 2.2, the condition that ``$R(S)$ is complete in $X$" is essential. For example,
Let $X=\Bbb{R}$ and $\mathcal{A}=M_2(\Bbb{C})$. Define $d\colon X\times X\rightarrow \mathcal{A}$ by
$d(x,y)=\left[
          \begin{array}{cc}
            |x-y| & 0 \\
            0 & k|x-y| \\
          \end{array}
        \right]$,
where $k>0$ is a constant.
Then $(X,\mathcal{A},d)$ is a complete $C^*$-algebra-valued metric space.
Define two mappings $T$ and $S$ by the following way
\begin{eqnarray*}
Tx=\left\{ \begin{array}{cc}
      \frac{k}{2}x & x\neq 0, \\[5pt]
      1 & x=0,
    \end{array}
    \right.
\ \mbox{and}\ \
Sx=\left\{ \begin{array}{cc}
      kx & x\neq 0, \\[5pt]
      2 & x=0.
    \end{array}
    \right.
    \end{eqnarray*}
One can verify that
\begin{eqnarray*}
d(Tx,Ty)\preceq a^*d(Sx,Sy)a,
\end{eqnarray*}
where $a=\left[
      \begin{array}{cc}
            \frac{\sqrt{2}}{2} & 0 \\
            0 &  \frac{\sqrt{2}}{2} \\
          \end{array}
        \right] \in\mathcal A$ and $\|a\|=\frac{\sqrt{2}}{2}\in(0,1)$.
And, $R(T)\subseteq R(S)$. But $R(S)$ is not complete in $X$. We can compute that $T$ and $S$ do not have a point of coincidence in $X$.
\end{example}

\begin{theorem}
Let $(X,\mathcal{A},d)$ be a complete $C^*$-algebra-valued metric space.
Suppose that two mappings $T,S\colon X\rightarrow X$ satisfy
\begin{equation}\label{2-4}
d(Tx,Ty)\preceq ad(Tx,Sx)+ad(Ty,Sy), \ \mbox{for\ any}\ x,y\in X,
\end{equation}
where $a\in \mathcal{A}_{+}'$ with $\|a\|<\frac{1}{2}$. If $R(T)$ is contained in $R(S)$ and $R(S)$ is complete in $X$,
then $T$ and $S$ have a unique point of coincidence in $X$. Furthermore,
if $T$ and $S$ are weakly compatible,
$T$ and $S$ have a unique
common fixed point in $X$.
\end{theorem}

\begin{proof}
Similar to Theorem 2.2, construct a sequence $\{x_n\}_{n=1}^\infty$ in $X$ such that $Sx_{n}=Tx_{n-1}$. Then from (\ref{2-4}),
\begin{eqnarray*}
    \begin{array}{rcl}
d(Sx_{n+1},Sx_{n})&=&d(Tx_{n},Tx_{n-1})\\[5pt]
      &\preceq&ad(Tx_{n},Sx_{n})+ad(Tx_{n-1},Sx_{n-1})\\[5pt]
      &=&ad(Sx_{n+1},Sx_n)+ad(Sx_{n},Sx_{n-1}),
 \end{array}
\end{eqnarray*}
which implies that
\begin{equation*}
(1-a)d(Sx_{n+1},Sx_{n})\preceq ad(Sx_{n},Sx_{n-1}).
\end{equation*}
Since $\|a\|<\frac{1}{2}$, then $1-a$ is invertible, and can be expressed as $(1-a)^{-1}=\sum\limits_{n=0}^\infty a^n$, which together with
$a\in \mathcal{A}_{+}'$ can yields $(1-a)^{-1}\in \mathcal{A}_{+}'$. By Lemma \ref{L1} (2), we know
\begin{equation}\label{P1}
d(Sx_{n+1},Sx_{n})\preceq bd(Sx_{n},Sx_{n-1}),
\end{equation}
where $b=(1-a)^{-1}a\in \mathcal{A}_{+}'$ with $\|b\|<1$.
Now, by induction and Lemma \ref{L1} (2), we can get
\begin{equation*}
d(Sx_{n+1},Sx_{n})\preceq b^nd(Sx_{1},Sx_{0}).
\end{equation*}
For $n>m$,
\begin{eqnarray*}
    \begin{array}{rcl}
d(Sx_{n},Sx_{m})
      &\preceq&d(Sx_{n},Sx_{n-1})+d(Sx_{n-1},Sx_{n-2})
      +\cdots d(Sx_{m+1},Sx_{m})\\[5pt]
      &\preceq&(b^{n-1}+b^{n-2}+\cdots+b^{m})d(Sx_{1},Sx_{0})\\[5pt]
      &\preceq& \|b^{n-1}+b^{n-2}+\cdots+b^{m}\|\|d(Sx_{1},Sx_{0})\|1_{\mathcal A}\\[5pt]
      &\preceq& \|b\|^{n-1}+\|b\|^{n-2}+\cdots+\|b\|^{m}\|d(Sx_{1},Sx_{0})\|1_{\mathcal A}\\[5pt]
      &=&\frac{\|b\|^{m}}{1-\|b\|}\|d(Sx_{1},Sx_{0})\|1_{\mathcal A}.
 \end{array}
\end{eqnarray*}
Hence $\{Sx_n\}_{n=0}^\infty$ is a Cauchy sequence in $R(S)$. The completion of $R(S)$ implies there is $q\in X$ such that $\lim\limits_{n\rightarrow\infty}Sx_n=Sq$.

Again, by (\ref{P1}), we have
\begin{equation*}
d(Sx_{n},Tq)=d(Tx_{n-1},Tq)\preceq bd(Sx_{n-1},Sq),
\end{equation*}
which implies that $\lim\limits_{n\rightarrow\infty}Sx_n=Tq$. The uniqueness of a limit in $C^*$-algebra-valued metric spaces tells us that $Tq=Sq$ (Lemma \ref{L1} (3)). Hence $T$ and $S$ have a point of coincidence in $X$. In the following we will show the uniqueness of points of coincidence.
To do this, we assume that there is $p\in X$ such that $Tp=Sp$.
Using (\ref{P1}), we obtain
\begin{equation*}
d(Sp,Sq)=d(Tp,Tq)\preceq ad(Tp,Sp)+ad(Tq,Sq),
\end{equation*}
which shows that
$\|d(Sp,Sq)\|=0$, and then $Sp=Sq$. It follows from Lemma \ref{L2} that $T$ and $S$ have a unique common fixed point in $X$.
\end{proof}

\begin{theorem}
Let $(X,\mathcal{A},d)$ be a complete $C^*$-algebra-valued metric space.
Suppose that two mappings $T,S\colon X\rightarrow X$ satisfy
\begin{equation}\label{2-5}
d(Tx,Ty)\preceq ad(Tx,Sy)+ad(Sx,Ty), \ \mbox{for\ any}\ x,y\in X,
\end{equation}
where $a\in \mathcal{A}_{+}'$ with $\|a\|<\frac{1}{2}$. If $R(T)$ is contained in $R(S)$ and $R(S)$ is complete in $X$,
then $T$ and $S$ have a unique point of coincidence in $X$. Furthermore,
if $T$ and $S$ are weakly compatible,
$T$ and $S$ have a unique
common fixed point in $X$.
\end{theorem}

\begin{proof}
Similar to Theorem 2.2, construct a sequence $\{x_n\}_{n=1}^\infty$ in $X$ such that $Sx_{n}=Tx_{n-1}$. Then from (\ref{2-5}),
\begin{eqnarray*}
    \begin{array}{rcl}
d(Sx_{n+1},Sx_{n})&=&d(Tx_{n},Tx_{n-1})\\[5pt]
      &\preceq&ad(Tx_{n},Sx_{n-1})+ad(Sx_{n},Tx_{n-1})\\[5pt]
      &=&ad(Sx_{n+1},Sx_{n-1})+ad(Sx_{n},Sx_{n})\\[5pt]
      &\preceq&ad(Sx_{n+1},Sx_{n})+ad(Sx_{n},Sx_{n-1}),
 \end{array}
\end{eqnarray*}
which implies that
\begin{equation*}
d(Sx_{n+1},Sx_{n})\preceq bd(Sx_{n},Sx_{n-1}),
\end{equation*}
where $b=(1-a)^{-1}a\in \mathcal{A}_{+}'$ with $\|b\|<1$.
The same argument in Theorem 2.3, we know $T$ and $S$ have a point of coincidence $Tq$ in $X$. In the following we will show the uniqueness of points of coincidence.
To do this, we assume that there is $p\in X$ such that $Tp=Sp$.
Using (\ref{2-5}), we obtain
\begin{equation*}
d(Sp,Sq)=d(Tp,Tq)\preceq ad(Tp,Sq)+ad(Sp,Tq)=2ad(Sp,Sq),
\end{equation*}
which together with $\|2a\|<1$ yields that
$\|d(Sp,Sq)\|=0$, and then $Sp=Sq$. It follows from Lemma \ref{L2} that $T$ and $S$ have a unique common fixed point in $X$.
\end{proof}

In Theorem 2.4, we choose $S=\mbox{id}_X$, then $R(S)=X$, and $T$ is weakly compatible with $S$. Moreover, we have the following consequence, which can also be seen in \cite{Ma}.

\begin{corollary}
Let $(X,\mathcal{A},d)$ be a complete $C^*$-algebra-valued metric space.
Suppose that the mapping $T\colon X\rightarrow X$ satisfies
\begin{equation*}
d(Tx,Ty)\preceq ad(Tx,y)+ad(Ty,x), \ \mbox{for\ any}\ x,y\in X,
\end{equation*}
where $a\in \mathcal{A}_{+}'$ with $\|a\|<\frac{1}{2}$,
then $T$ have a unique point in $X$.
\end{corollary}

Fixed point theorems for operators in metric spaces are widely investigated and have found various applications in differential and integral equations \cite{Amin,Harja}.
As an application, let us consider the following system of integral equations
\begin{eqnarray}
    \begin{array}{rcl}
x(t)&=&\int_E K_1(t,s,x(s))ds+g(t), \ \ t\in E,\\[5pt]
x(t)&=&\int_E K_2(t,s,x(s))ds+g(t), \ \ t\in E,
 \end{array}
\end{eqnarray}
where $E$ is a Lebesgue measurable set and $m(E)<\infty$.

\begin{theorem}
Assume that the following hypotheses hold

(1)\ $K_1\colon E\times E\times\Bbb{R}\rightarrow\Bbb{R}$,
$K_2\colon E\times E\times\Bbb{R}\rightarrow\Bbb{R}$ are integrable, and $g\in L^\infty(E)$;

(2)\ there exist $k\in (0,1)$ and a continuous function $\varphi\colon E\times E\rightarrow\Bbb{R}^+$ such that
$$|K_1(t,s,u)-K_2(t,s,v)|\leq k\varphi(t,s)|u-v|$$ for $t,s\in E$ and $u,v\in\Bbb{R}$;

(3)\ $\sup\limits_{t\in E}\int_E\varphi(t,s)ds\leq 1$.\\
Then the integral equations (2.7) have a unique common solution in $L^\infty(E)$.
\end{theorem}

\begin{proof}
Let $X=L^\infty(E)$ be the set of essentially bounded measurable functions on $E$ and $B(L^2(E))$ be the set of bounded linear operators on a Hilbert space $L^2(E)$.
Consider $d\colon X\times X\rightarrow B(L^2(E))$ defined by
$d(f,g)=M_{|f-g|}$,
where $M_{|f-g|}$ is the multiplication operator on $L^2(E)$. Then $(X,B(L^2(E)),d)$ is a complete $C^*$-algebra-valued metric space.

Define $T,S\colon X\rightarrow X$ by
\begin{eqnarray*}
    \begin{array}{rcl}
T(x(t))&=&\int_E K_1(t,s,x(s))ds+g(t), \ \ t\in E,\\[5pt]
S(x(t))&=&\int_E K_2(t,s,x(s))ds+g(t), \ \ t\in E.
 \end{array}
\end{eqnarray*}

Notice that
\begin{eqnarray*}
    \begin{array}{rcl}
\|d(Tx,Sy)\|&=&\sup\limits_{\|\varphi\|=1}(M_{|Tx-Sy|}\varphi,\varphi)\\[5pt]
&=&\sum\limits_{\|\varphi\|=1}\int_E|\int_E (K_1(t,s,x(s))-K_2(t,s,y(s)))ds|\varphi(t)\overline{\varphi(t)}dt\\[5pt]
&\leqslant&\sup\limits_{\|\varphi\|=1}\int_E\int_E |K_1(t,s,x(s))-K_2(t,s,y(s))|ds\varphi(t)\overline{\varphi(t)}dt\\[5pt]
&\leqslant&\sup\limits_{\|\varphi\|=1}k\int_E\int_E \varphi(t,s)|x(s)-y(s)|ds\varphi(t)\overline{\varphi(t)}dt\\[5pt]
&\leqslant&\sup\limits_{\|\varphi\|=1}k\int_E\int_E \varphi(t,s)ds|\varphi(t)|^2dt\|x-y\|_\infty\\[5pt]
&\leqslant&k\sup\limits_{t\in E}\int_E\varphi(t,s)ds
\sup\limits_{\|\varphi\|=1}
\int_E |\varphi(t)|^2dt\|x-y\|_\infty\\[5pt]
&\leqslant&k\|d(x,y)\|.
 \end{array}
\end{eqnarray*}
Thus it is verified that the mappings $T$ and $S$ satisfy all the conditions of Corollary 2.1, and then $T$ and $S$ have a unique common fixed point, which is equivalent to that
 the integral equations (2.7) have a unique common solution in $L^2(E)$.
\end{proof}

 \end{document}